\documentclass{amsart}
\usepackage{times}

\newtheorem{theorem}{Theorem}[section]
\newtheorem{lemma}[theorem]{Lemma}
\newtheorem{proposition}[theorem]{Proposition}

\theoremstyle{definition}

\theoremstyle{remark}

\numberwithin{equation}{section}

\begin{document}

\title[Fock space]
{Some Characterizations for Composition Operators on the Fock Space}

\author[L. Jiang]{Liangying Jiang}
\address{
Liangying Jiang\\
Department of Statistics and  Mathematics,
Shanghai Finance University,
Shanghai 201209,
P. R. China}
\email{liangying1231@163.com, jiangly@sfu.edu.cn}

\author[G. T. Prajitura]
{Gabriel T. Prajitura}
\address{
Gabriel T. Prajitura\\
Department of Mathematics,
SUNY Brockport,
Brockport,
\linebreak
NY 14420,
USA}
\email{gprajitu@brockport.edu}

\author[R. Zhao]
{Ruhan Zhao}
\address{
Ruhan Zhao\\
Department of Mathematics,
SUNY Brockport,
Brockport, NY 14420,
USA}
\email{rzhao@brockport.edu}



\date{21 August 2016.
\newline \indent This  is supported by the National
Natural Science Foundation of China (No.11571256 and No.11601400)}

\begin{abstract}
We study composition operators on the Fock spaces $\mathcal{F}^2_\alpha(\mathbb{C}^n)$, problems considered
include the essential norm, normality, spectra, cyclicity and
membership in the Schatten classes. We give  perfect answers for these basic properties , which present lots of  different  characterizations 
with the composition operators on 
the  Hardy space or  the weighted  Bergman spaces.
\end{abstract}

\maketitle

\section{Introduction and preliminaries}\label{Sec1}

Let $z=(z_1, \ldots, z_n)$  and $w=(w_1, \ldots, w_n)$ be the points in  $\mathbb{C}^n$, the inner product is $$\langle z, w\rangle=\sum_{j=1}^n z_j \overline{w_j}$$ and $|z|=\sqrt{\langle z, z\rangle}$.

For any $\alpha>0$, consider the Gaussian probability measure $$dv_\alpha(z)=\biggl(\frac{\alpha}{\pi}\biggr)^ne^{-\alpha|z|^2}dv(z)$$
on $\mathbb{C}^n$, where  $dv$ is the Lebesgue volume  measure on  $\mathbb{C}^n$.
The Fock space $\mathcal{F}^2_\alpha(\mathbb{C}^n)$ consists of all holomorphic
functions $f$ on $\mathbb{C}^n$  with
$$\|f\|^2_\alpha\equiv \int_{\mathbb{C}^n}|f(z)|^2dv_\alpha(z)
<\infty.
$$
Thus, $\mathcal{F}^2_\alpha(\mathbb{C}^n)$ is a Hilbert space with the following inner product
$$\langle f, g\rangle_\alpha=\int_{\mathbb{C}^n}f(z) \overline{g(z)}dv_\alpha(z).$$
Its reproducing kernels are given by
$$ K(z, w)=K_w(z)=e^{\alpha\langle z,w\rangle}$$ with $\|K_w\|^2_\alpha=\exp(\alpha|z|^2)$.

For a given holomorphic mapping $\varphi:\mathbb{C}^n\rightarrow \mathbb{C}^n$,
the composition operator $C_\varphi$ on the Fock spaces $\mathcal{F}^2_\alpha(\mathbb{C}^n)$ is
defined by $C_\varphi(f)=f\circ\varphi$.

In 2003, Carswell et al. \cite{CMS} first studied  composition operators on the classical Fock space $\mathcal{F}^2_\alpha(\mathbb{C}^n)$ when $\alpha=1/2$, usually denoted by $\mathcal{F}^2(\mathbb{C}^n)$.
They found the following information.

{\bf Theorem A.}
{\it Suppose $\varphi:\mathbb{C}^n\rightarrow \mathbb{C}^n$ is a holomorphic mapping.
\begin{itemize}
\item[(a)] $C_\varphi$ is bounded on $\mathcal{F}^2(\mathbb{C}^n)$ if and only if $\varphi(z)=Az+B$,
where $A$ is an $n\times n$ matrix and $B$ is
an $n\times 1$ vector. Furthermore, $\|A\|\le 1$,
and $\langle A\zeta,B\rangle=0$ if $|A\zeta|=|\zeta|$ for some $\zeta\in \mathbb{C}^n$.
\item[(b)] $C_\varphi$ is compact on $\mathcal{F}^2(\mathbb{C}^n)$ if and only if
$\varphi(z)=Az+B$, where $\|A\|< 1$.
\end{itemize}}

From this, we know that only a class of linear mappings of $\mathbb{C}^n$
can induce bounded composition operators on $\mathcal{F}^2(\mathbb{C}^n)$. In fact, this is the same on the
 Fock space $\mathcal{F}^2_\alpha(\mathbb{C}^n)$ for any $\alpha>0$. So, it is natural to
look forward  that more properties of bounded composition operators
on $\mathcal{F}^2_\alpha(\mathbb{C}^n)$ could be completely characterized.

In recent years, the study of composition operators and weighted composition operators on  Fock spaces has attracted a lot of attention (\cite{CCK}, \cite{CIK}, \cite{SS09},  \cite{U07}, \cite{LZ}).
However, some basic problems about composition operators on the Fock spaces $\mathcal{F}^2_\alpha(\mathbb{C}^n)$ are still open. Only in the setting of
 complex plane $\mathbb{C}$, Guo and Izuchi \cite{GI} described   some properties of composition operators
on Fock type spaces, including spectra, cyclicity and
connected components of the set of composition operators.

In this work, we try to investigate the basic operator  properties of composition operators
on  $\mathcal{F}^2_\alpha(\mathbb{C}^n)$, and  find that some behaviors are so distinctive. For simplicity, we will discuss our results on the
classical Fock space  $\mathcal{F}^2(\mathbb{C}^n)$. But all results hold on  Fock  spaces $\mathcal{F}^2_\alpha(\mathbb{C}^n)$ for all $\alpha>0$.

Let $\varphi$ be a holomorphic mapping of $\mathbb{C}^n$, which induces a bounded composition operator $C_\varphi$ on  $\mathcal{F}^2(\mathbb{C}^n)$.
First, we will calculate the essential norm $\|C_\varphi\|_e$ in Section 2. When $C_\varphi$ is not compact, we find that $\|C_\varphi\|_e=\|C_\varphi\|$.
In this section, we  also discuss the membership in the Schatten classes and
obtain that all compact composition operators belong to the Schatten $p$-class $S_p$ for all $0<p<\infty$.
Section 3 is devoted to  describe the normality of $C_\varphi$.  It is  interesting that when $C_\varphi$ is hypernormal on $\mathcal{F}^2(\mathbb{C}^n)$,
then $C_\varphi$ must be normal. Moreover, there are no non-trivial essentially normal composition operators on $\mathcal{F}^2(\mathbb{C}^n)$. All above results are different to that when  $C_\varphi$ acting on other classical function spaces, such as the Hardy space and the weighted Bergman spaces.

In Section 4, we completely give  the spectrum of  $C_\varphi$ on  $\mathcal{F}^2(\mathbb{C}^n)$, that is, $$\sigma(C_\varphi)=\overline{\{\lambda_1^{\gamma_1} \cdots \lambda_n^{\gamma_n}: (\gamma_1, \ldots, \gamma_n) \in \mathbb{N}^n\}},$$
where $\lambda_1, \ldots, \lambda_n$ are the eigenvalues of the matrix $A$ (which described in Theorem A).
Finally, the cyclicity of composition operators will be studied in Section 5.
We observe that all bounded composition operators are not supercyclic  on $\mathcal{F}^2(\mathbb{C}^n)$, and give
 a necessary and sufficient condition for $C_\varphi$ to be cyclic when $\varphi$ is  unitary. Thus, it remains that, whether $C_\varphi$ is cyclic is unknown if $\varphi$ is univalent, not unitary.

\section { Essential norm and Schatten class}

When  $C_\varphi$  is bounded on the Fock space $\mathcal{F}^2(\mathbb{C}^n)$, by Theorem A, we have $\varphi(z)=Az+B$.
The norm of $C_\varphi$ has been obtained by  Carswell et al. \cite{CMS} as the following:
$$
\|C_\varphi\|=\exp\biggl(\frac{1}{4}(|z_0|^2-|Az_0|^2+|B|^2)\biggr),
$$
where $z_0$ is any solution to $(I-A^*A)z=A^*B$.

In this section, we first determine the essential norm of  $C_\varphi$, which
 defined as:
$$\|C_\varphi\|_e=\|C^\ast_\varphi\|_e=\inf\{\|C^\ast_\varphi-F\|: F\ \mbox{is compact on} \ \mathcal{F}^2(\mathbb{C}^n)\}.$$

\begin{theorem}\label{the 2.1}
Let $\varphi(z)=Az+B$ and $C_\varphi$ be bounded, not compact, on $\mathcal{F}^2(\mathbb{C}^n)$.
Then the essential norm of $C_\varphi$ is
$$
\|C_\varphi\|_e=\|C_\varphi\|=\exp\biggl(\frac{1}{4}\langle \varphi(z_0), B\rangle\biggr),
$$
where $z_0$ is any solution to $(I-A^*A)z=A^*B$.
\end{theorem}

\begin{proof} Since $C_\varphi$ is not compact, we get that $\|A\|=1$ from Theorem A. Let $A=V\Sigma W$ be the singular value decomposition of $A$,
where $V, W$ are unitary and $\Sigma=\mbox{diag}(\sigma_1,\sigma_2,\ldots,\sigma_n)$ with $\sigma_1\ge \sigma_2\ge \cdots\ge \sigma_n\ge 0$ being
the singular values of $A$. Set $j=\max\{r: \sigma_r=1\}$. It is clear that $j\ge 1$ from  $\sigma_1=\|A\|=1$.

Let $\psi(z)=\Sigma z+B'$ be the normalization of $\varphi$ (see Proposition 1 in \cite{CMS}), that is, $\varphi=V\circ \psi\circ W$,
where $B'=V^* B$. Denote $C_W$ and $ C_V$ the
unitary operators respectively given by $C_W(f)=f\circ W$ and $C_V(f)=f\circ V$. This means  $C_\varphi=C_WC_\psi C_V$.
Using Corollary 1 of \cite{CMS}, we get that  $\|C_\varphi\|=\|C_\psi\|$ and  $\|C_\varphi\|_e=\|C_\psi\|_e$.
So it suffices to calculate the essential norm of $C_\psi$.

Let $K_w$ be the reproducing kernels  for  $\mathcal{F}^2(\mathbb{C}^n)$.
The proof of Lemma 3 in \cite{CMS} together with (11) of \cite{CMS} gives
$$
\|C_\psi\|=\|C^*_\psi\|
=\sup\limits_{w\in \mathbb{C}^n}\frac{\|C_\psi^*(K_w)\|}{\|K_w\|}
=\sup\limits_{w\in \mathbb{C}^n}\frac{\|K_{\psi(w)}\|}{\|K_w\|},
$$
and $\|K_{\psi(w)}\|/\|K_w\|$ attains its maximum at points $w=(w_1,w_2,\ldots,w_n)\in \mathbb{C}^n$
which satisfy, for $j+1\le m\le n$,
$$
|w_m|=\frac{\sigma_m|b'_m|}{1-\sigma^2_m}\quad
\mbox{and}
\quad
\mbox{arg}\ w_m \  \mbox{chosen\ so\ that} \ w_mb'_m\ge 0,
$$
and arbitrary $w_m$ for $m\le j$, where $b'_m$ is the $m^{th}$ coordinate of $B'$.

On the other hand, the normalized reproducing kernel functions $k_w=K_w/\|K_w\|$ tend to $0$ weakly on $\mathcal{F}^2$ as $|w|\to \infty$, so that
\begin{eqnarray*}
\|C_\psi\|_e
&=&\|C^*_\psi\|_e \ge \limsup\limits_{|w|\to\infty}\frac{\|C_\psi^*(K_w)\|}{\|K_w\|}\\
&\ge & \limsup\limits_{\substack{|w'|\to\infty\\
w'=(w_1,\ldots,w_j)}}\biggl\{\frac{\|C_\psi^*(K_w)\|}{\|K_w\|}: \  |w_m|=\frac{\sigma_m|b'_m|}{1-\sigma^2_m} \ \\
&& \mbox{and} \ \mbox{arg}\ w_m \ \mbox{chosen\ so\ that}\
w_mb'_m\ge 0\  \mbox{for}\  j+1\le m\le n\biggr\}\\
&=& \sup\limits_{w\in \mathbb{C}^n}\frac{\|C_\psi^*(K_w)\|}{\|K_w\|}
=\|C_\psi\|.
\end{eqnarray*}
However,  we know that
$\|C_\psi\|_e\le \|C_\psi\|$. This together with the above discussion gives
$$
\|C_\psi\|_e=\|C_\psi\|= \sup\limits_{w\in \mathbb{C}^n}\frac{\|C_\psi^*(K_w)\|}{\|K_w\|}.
$$
Therefore, $\|C_\varphi\|_e=\|C_\psi\|_e= \|C_\psi\|=\|C_\varphi\|$.

Now, Theorem 4 of \cite{CMS} gives that
$$
\|C_\varphi\|=\exp\biggl(\frac{1}{4}(|z_0|^2-|Az_0|^2+|B|^2)\biggr),
$$
where $z_0$ is any solution to $(I-A^*A)z=A^*B$. In fact, $(I-A^*A)z_0=A^*B$ yields that
$$
|z_0|^2-|Az_0|^2
=\langle z_0, (I-A^*A)z_0 \rangle=\langle z_0, A^*B\rangle
=\langle  Az_0, B \rangle.
$$
It follows that{\setlength\arraycolsep{2pt}
\begin{eqnarray*}
\|C_\varphi\|_e &=& \|C_\varphi\|=\exp\biggl(\frac{1}{4}(|z_0|^2-|A z_0|^2+|B|^2)\biggr)
\\ &=&\exp\biggl(\frac{1}{4}(\langle A z_0, B\rangle+|B|^2)\biggr)
\\ &=&\exp\biggl(\frac{1}{4}\langle \varphi(z_0), B\rangle\biggr).
\end{eqnarray*}}
\end{proof}

Next, we will discuss Schatten class composition operators on the Fock space $\mathcal{F}^2(\mathbb{C}^n)$. Recall that if $T$ is  a compact
operator on a separable Hilbert space $H$, then there exist orthonormal sets $\{e_n\}$ and $\{\sigma_n\}$ in $H$ such that
$$Tx=\sum_n \lambda_n \langle x, e_n \rangle \sigma_n, \qquad x\in H,$$
where $\lambda_n$ is the  singular value of $T$, i.e. it is the eigenvalue of $|T|=(T^*T)^{1/2}$.

Given $0<p<\infty$,
if the sequence $\{\lambda_n\}$ belongs to $l^p$, we say that $T$ belongs to the Schatten $p$-class of $H$, denoted $S_p(H)$ or $S_p$,
and define the norm $$\|T\|_p=\biggl[\sum_n |\lambda_n|^p\biggr]^{1/p}.$$
Usually, $S_1$ is also called the trace class and $S_2$ is called the Hilbert-Schmidt class.

It is well known that if $T$ is compact on $H$, then $T\in S_p$ if and only if $T^*T\in S_p$ with $\|T\|^p_p=\|T^*T\|_{\frac{p}{2}}$. We refer to \cite{Zhu90} for more information about the Schatten classes.

For a Borel measure $\mu$ on $\mathbb{C}^n$, we define the Toeplitz operator $T_\mu$ on $\mathcal{F}^2(\mathbb{C}^n)$ as follows:
$$T_\mu(f)(z)=\biggl(\frac{1}{2\pi}\biggr)^n\int_{\mathbb{C}^n}f(w)K(z, w)e^{-\frac{|w|^2}{2}}d\mu(w).$$
The Berezin transform of $\mu$ is  defined:
{\setlength\arraycolsep{2pt}
\begin{eqnarray*}\widetilde{\mu}(z)&=&\langle T_\mu k_z, k_z\rangle
=\biggl(\frac{1}{2\pi}\biggr)^n\int_{\mathbb{C}^n}|k_z(w)|^2e^{-\frac{|w|^2}{2}}d\mu(w)
\\ &=&\biggl(\frac{1}{2\pi}\biggr)^n\int_{\mathbb{C}^n}e^{-\frac{1}{2}|z-w|^2}d\mu(w),
\end{eqnarray*}}where $k_z(w)=K(w,z)/\sqrt{K(z,z)}=e^{\frac{1}{2}\langle w, z\rangle-\frac{1}{4}|z|^2}$
are the normalized reproducing kernels of  $\mathcal{F}^2(\mathbb{C}^n)$.

\begin{theorem}\label{the 2.2}
Suppose that $C_\varphi$ is compact on $\mathcal{F}^2(\mathbb{C}^n)$,
then  $C_\varphi$ belongs to the Schatten $p$-class $S_p$ for all $0<p<\infty$.
\end{theorem}

\begin{proof}
First, for a function $f$ in $\mathcal{F}^2(\mathbb{C}^n)$, it is easy to check that
{\setlength\arraycolsep{2pt}
\begin{eqnarray*}C_\varphi^*C_\varphi f(z)&=&\langle C_\varphi^*C_\varphi f, K_z\rangle=\langle C_\varphi f, C_\varphi K_z\rangle
\\ &=&\biggl(\frac{1}{2\pi}\biggr)^n\int_{\mathbb{C}^n}f(\varphi(w)) K_z(\varphi(w))e^{-\frac{|w|^2}{2}}dv(w)
\\ &=&\biggl(\frac{1}{2\pi}\biggr)^n\int_{\mathbb{C}^n}f(\varphi(w)) K_z(\varphi(w))e^{-\frac{|\varphi(w)|^2}{2}}e^{\frac{1}{2}(|\varphi(w)|^2-|w|^2)}dv(w)
\\ &=&\biggl(\frac{1}{2\pi}\biggr)^n\int_{\mathbb{C}^n}f(u) K_z(u)e^{-\frac{|u|^2}{2}}d\mu(u)=T_\mu(f)(z),
\end{eqnarray*}}where $d\mu=d\nu\circ \varphi^{-1}$ with $d\nu(z)=e^{\frac{1}{2}(|\varphi(z)|^2-|z|^2)}dv(z)$.

Thus, for $0<p<\infty$, the composition operator $C_\varphi$ belongs to the Schatten $p$-class $S_p$
if and only if $T_\mu=C_\varphi^*C_\varphi$ belongs to $S_{p/2}$. By Theorem 2.7 and Theorem 3.2 in \cite{IVW15}, it is equivalent to that
$\widetilde{\mu}(z)$ is in $L^{p/2}(\mathbb{C}^n, dv)$.

Now, we calculate that
{\setlength\arraycolsep{2pt}
\begin{eqnarray*}
&& \int_{\mathbb{C}^n}\widetilde{\mu}(z)^{p/2}dv(z)
=\int_{\mathbb{C}^n}\biggl(\biggl(\frac{1}{2\pi}\biggr)^n\int_{\mathbb{C}^n}e^{-\frac{1}{2}|z-w|^2}d\mu(w)\biggr)^{p/2}dv(z)
\\ &=&\int_{\mathbb{C}^n} \biggl(\biggl(\frac{1}{2\pi}\biggr)^n\int_{\mathbb{C}^n}e^{-\frac{1}{2}|z-\varphi(u)|^2}
e^{\frac{1}{2}(|\varphi(u)|^2-|u|^2)}dv(u)\biggr)^{p/2}dv(z)
\\ &=&\int_{\mathbb{C}^n}e^{-\frac{p}{4}|z|^2} \biggl(\biggl(\frac{1}{2\pi}\biggr)^n\int_{\mathbb{C}^n}e^{Re \langle \varphi(u), z\rangle}
e^{-\frac{1}{2}|u|^2}dv(u)\biggr)^{p/2}dv(z).
\end{eqnarray*}}

If $C_\varphi$ is compact, then $\varphi(z)=Az+B$ with $\|A\|<1$ from Theorem A. On the other hand, Lemma 3 in \cite{DZ08} gives that
$$\int_{\mathbb{C}^n}|e^{s\langle z, a\rangle}|dv_\alpha(z)=e^{s^2|a|^2/4\alpha}\eqno(2.1)$$ for all $a\in\mathbb{C}^n$, where $s$ is real.
Therefore, using the formula (2.1) twice,
{\setlength\arraycolsep{2pt}
\begin{eqnarray*}
&& \int_{\mathbb{C}^n}\widetilde{\mu}(z)^{p/2}dv(z)
\\ &=& \int_{\mathbb{C}^n}e^{-\frac{p}{4}|z|^2} \biggl(\biggl(\frac{1}{2\pi}\biggr)^n\int_{\mathbb{C}^n}e^{Re \langle Au+B, z\rangle}
e^{-\frac{1}{2}|u|^2}dv(u)\biggr)^{p/2}dv(z)
\\ &=& \int_{\mathbb{C}^n}|e^{\frac{p}{2}\langle z, B\rangle}|e^{-\frac{p}{4}|z|^2}
 \biggl(\biggl(\frac{1}{2\pi}\biggr)^n\int_{\mathbb{C}^n}|e^{ \langle u, A^*z\rangle}|
e^{-\frac{1}{2}|u|^2}dv(u)\biggr)^{p/2}dv(z)
\\ &=& \int_{\mathbb{C}^n}|e^{\frac{p}{2}\langle z, B\rangle}|e^{-\frac{p}{4}|z|^2}(e^{|A^*z|^2/2})^{p/2}dv(z)
\\ &\le & \int_{\mathbb{C}^n}|e^{\frac{p}{2}\langle z, B\rangle}|e^{-\frac{p}{4}|z|^2}e^{\frac{p}{4}\|A^*\|^2|z|^2}dv(z)
\\ &=& \int_{\mathbb{C}^n}|e^{\frac{p}{2}\langle z, B\rangle}|e^{-\frac{p}{4}(1-\|A\|^2)|z|^2}dv(z)
\\ & \le & C e^{\frac{p|B|^2}{1-\|A\|}}<\infty,
\end{eqnarray*}}where $C>0$ is a constant depending on $\|A\|$. Thus, $\widetilde{\mu}(z)\in L^{p/2}(\mathbb{C}^n, dv)$
and  $C_\varphi$ is in the Schatten $p$-class $S_p$ for any $0<p<\infty$.

\end{proof}

In fact,  this result has been obtained by Du \cite{Du11} and Schatten class weighted composition operators on the Fock spaces have been studied by many
authors (see \cite{WA15}, \cite{TM13}, \cite{TM15}, \cite{Ue07}). However, our different proof leads to the following interest fact:
If $C_\varphi$ is compact on  $\mathcal{F}^2(\mathbb{C}^n)$,  the conditions
$$\int_{\mathbb{C}^n}\|C_\varphi k_z\|^pdv(z)<\infty\quad \mbox{and} \quad \int_{\mathbb{C}^n}\|C^*_\varphi k_z\|^pdv(z)<\infty\eqno(2.2)$$
are equivalent. While, in the setting of the Bergman space, Xia \cite{JXia03} found a self-mapping $\varphi$ of the unit disc $D$, which induces a
composition operator satisfying
$$\int_D\|C^*_\varphi k_z\|^p\frac{dv(z)}{1-|z|^2}<\infty\quad \mbox{but} \quad \int_D\|C_\varphi k_z\|^p\frac{dv(z)}{1-|z|^2}=\infty.$$

Theorem 2.2 and the above discussion reflect that there are lots of  different characteristics for composition
operators on the  Fock spaces. As we know, many examples show that compact composition operators on the Bergman space  maybe are
not in any of the Schatten classes (see \cite{CC91}, \cite{BLo98}, \cite{YZhu01}).

Next, we will prove how the conditions in (2.2) are equivalent when $C_\varphi$ is compact on $\mathcal{F}^2(\mathbb{C}^n)$. The proof of Theorem 2.2 gives that
{\setlength\arraycolsep{2pt}
\begin{eqnarray*}
\int_{\mathbb{C}^n}\|C_\varphi k_z\|^pdv(z)&=& \int_{\mathbb{C}^n}\langle C^*_\varphi C_\varphi k_z, k_z\rangle^{\frac{p}{2}}dv(z)
\\ &=& \int_{\mathbb{C}^n}\langle T_\mu k_z, k_z\rangle^{\frac{p}{2}}dv(z)=\int_{\mathbb{C}^n}\widetilde{\mu}(z)^{\frac{p}{2}}dv(z)
\\ & \le & C e^{\frac{p|B|^2}{1-\|A\|}}<\infty.
\end{eqnarray*}}On the other hand, using the fact $C_\varphi^* K_z=K_{\varphi(z)}$ and the equation (2.1), we obtain
{\setlength\arraycolsep{2pt}
\begin{eqnarray*}
\int_{\mathbb{C}^n}\|C^*_\varphi k_z\|^pdv(z)&=& \int_{\mathbb{C}^n}\biggl(\frac{\|K_{\varphi(z)}\|}{\|K_z\|}\biggr)^pdv(z)
\\ &=& \int_{\mathbb{C}^n} e^{\frac{p}{4}(|\varphi(z)|^2-|z|^2)}dv(z)
=\int_{\mathbb{C}^n} e^{\frac{p}{4}(|Az+B|^2-|z|^2)}dv(z)
\\ &\le & e^{\frac{p}{4}|B|^2} \int_{\mathbb{C}^n}|e^{\frac{p}{2}\langle z, A^*B\rangle}|e^{-\frac{p}{4}(1-\|A\|^2)|z|^2}dv(z)
\\ & \le & C e^{\frac{p}{4}|B|^2} e^{\frac{p|A^*B|^2}{1-\|A\|}}<\infty.
\end{eqnarray*}}

\section { Normal composition operators}

In this section, we will characterize normal  composition operators  on the Fock space $\mathcal{F}^2(\mathbb{C}^n)$.
This property for weighted composition operators has been investigated  in some articles \cite{SP12}, \cite{SJ11}.
But we  will try to reveal a perfect nature of normal composition operators on the Fock spaces.

First, we give a complete characterization
for normal composition operators. Using this,  we find an interesting result,  that is, hypernormal  composition operators must be normal.
 Furthermore, we prove that
if $C_\varphi$ is essentially normal then  $C_\varphi$ must be compact or normal. Here,  $C_\varphi$ is called essentially normal if the commutator
$[C^*_\varphi, C_\varphi]=C^*_\varphi C_\varphi-C_\varphi C^*_\varphi$
is compact.

\begin{theorem}\label{the 3.1}
Let $\varphi$ be a holomorphic mapping of $\mathbb{C}^n$.
Assume that $C_\varphi$ is bounded  on $\mathcal{F}^2(\mathbb{C}^n)$.
Then $C_\varphi$ is normal if and only if  $\varphi(z)=Az$ with $A^*A=AA^*$.
\end{theorem}

\begin{proof}
This result can be shown using the same idea as Theorem 8.2 of \cite{CM1}. Here, we give a simple argument.

First, since $C_\varphi$ is bounded on $\mathcal{F}^2(\mathbb{C}^n)$, by Theorem A, we have $\varphi(z)=Az+B$. Now, assume that $C_\varphi$ is normal (then it is hyponormal)  on $\mathcal{F}^2$. Then $$1=\|C_\varphi 1\|^2\ge\|C_\varphi^* 1\|^2=\|C_\varphi^* K_0\|^2=\|K_{\varphi(0)}\|^2=e^{|\varphi(0)|^2/2},$$
which  implies $\varphi(0)=0$. So we get  $B=0$ and $\varphi(z)=Az$.

Lemma 2 in \cite{CMS} shows that the adjoint $C_\varphi^*=C_\tau$ with $\tau(z)=A^*z$. Since  $C_\varphi$ is normal, we have
$$C_{\varphi\circ\tau}=C_\tau C_\varphi=C_\varphi^*C_\varphi=C_\varphi C_\varphi^*=C_\varphi C_\tau=C_{\tau\circ \varphi}.$$
It follows that $AA^*z=\varphi\circ\tau(z)=\tau\circ \varphi(z)=A^*Az$ for any $z\in \mathbb{C}^n$. Therefore, we deduce $AA^*=A^*A$. The other direction is obvious, so we complete the proof.
\end{proof}

{\bf Remark.} From this result, we may  find  that Proposition 2.4 in \cite{SJ11} is not true, because the authors deduced that if  $C_\varphi$ is normal on  $\mathcal{F}^2(\mathbb{C}^n)$, then $\varphi$ is univalent.
\\ \par
Next, we also present  the following  interesting result.

\begin{proposition}\label{prop 3.2} Assume that  $\varphi:\mathbb{C}^n\rightarrow \mathbb{C}^n$ is  a holomorphic mapping
and  $C_\varphi$ is bounded on $\mathcal{F}^2(\mathbb{C}^n)$. If $C_\varphi$ is  hyponormal, then it is  normal.
\end{proposition}

\begin{proof}
Suppose that
 $C_\varphi$  is hyponormal on $\mathcal{F}^2(\mathbb{C}^n)$. The preceding argument in the proof of Theorem 3.1 shows that $\varphi(z) = Az$ for some
$n\times n$ matrix $A$ with $\|A\|\le 1$.

Let $v_1$ be an eigenvector of $A^*$ with corresponding eigenvalue $\overline{\lambda_1}$. Define $f\in \mathcal{F}^2(\mathbb{C}^n)$ by
$f(z) = \langle z, v_1\rangle$. Because $C_\varphi-\lambda_1I$ is also hyponormal, we have
$$\|(C_\varphi-\lambda_1I)f\|\ge \|(C_\varphi-\lambda_1I)^*f\|. \eqno(3.1)$$
However, for each $z$,
$$((C_\varphi-\lambda_1I)f)(z)=f(\varphi(z))-\lambda_1f(z)=\langle(A-\lambda_1I)z, v_1\rangle=\langle z, 0\rangle=0;$$
that is, $(C_\varphi-\lambda_1I)f$ is the zero function. Hence, by (3.1), $(C_\varphi-\lambda_1I)^*f$ is also the zero
function. This means that for each $z$,
$$0=((C_\varphi-\lambda_1I)^*f)(z)=\langle(A^*-\overline{\lambda_1}I)z, v_1\rangle=\langle z, (A-\lambda_1)v_1\rangle,$$
so that $(A-\lambda_1)v_1=0$, where we have used $C_\varphi^*=C_\tau$ with $\tau(z)=A^*z$.
 Thus $v_1$ is also an eigenvector for $A$ (with eigenvalue $\lambda_1$).

We see that the   subspace $W_1$ of $\mathbb{C}^n$ spanned by $v_1$ is reducing for $A^*$.
This gives  $A^*W_1^\perp\subseteq W_1^\perp$. Now,
 we apply the argument of the preceding paragraph
a second time, starting with an eigenvector $v_2\in W_1^\perp$
for $A^*$ with corresponding
eigenvalue $\overline{\lambda_2}$. Then, we obtain $A v_2=\lambda_2 v_2$. If $n = 2$, we are done,  $A^*$ and $A$ commute
on the basis $\{v_1, v_2\}$ of $\mathbb{C}^n$. Otherwise,  notice that the subspace  $W_2$ of  $\mathbb{C}^n$ spanned
by $v_1$ and $v_2$ is reducing for $A^*$  and  $A^*W_2^\perp\subseteq W_2^\perp$. Thus, we can again apply the
argument of the preceding paragraph to obtain a vector $v_3\in W_2^\perp$, which is both an
eigenvector for $A^*$ and for $A$. We can continue this process until we get a basis of $\mathbb{C}^n$ on which $A^*$ and $A$ commute. Hence, by Theorem 3.1,
we know that $C_\varphi$ is  normal on $\mathcal{F}^2(\mathbb{C}^n)$.
\end{proof}

Finally,  on the Fock space $\mathcal{F}^2(\mathbb{C}^n)$,
we observe that only compact and normal composition operators can be essentially normal.

\begin{theorem}\label{the 3.3} Suppose that $\varphi$ is a holomorphic mapping of $\mathbb{C}^n$ and
 $C_\varphi$ is bounded on $\mathcal{F}^2(\mathbb{C}^n)$.
Then $C_\varphi$ is essentially normal
if and only if  $C_\varphi$ is  either compact or normal.
\end{theorem}

\begin{proof}
Compact and normal operators are trivially essentially normal.
We only need to prove the other direction.

Assume that $C_\varphi$ is essentially normal,
i.e. the commutator $[C^*_\varphi, C_\varphi]$ is compact.
Let $k_p=K_p/\|K_p\|$ be the normalized reproducing kernel at $p\in \mathbb{C}^n$.
Note that the sequence $\{k_p\}$ tends to zero weakly on $\mathcal{F}^2(\mathbb{C}^n)$ as $|p|\to \infty$, so we have
$$
\limsup\limits_{|p|\to \infty}\|[C^*_\varphi, C_\varphi]k_p\|=0.
$$

Since  $C_\varphi$ is bounded on $\mathcal{F}^2(\mathbb{C}^n)$, we get $\varphi(z)=Az+B$,
where $A$ and $B$ are described as Theorem A. Moreover, Lemma 2 of \cite{CMS} gives $C^*_\varphi=M_{K_B}C_\tau$  with $\tau(z)=A^*z$.
Thus, for each $p\in  \mathbb{C}^n$,
$$
C_\varphi K_p= (M_{K_B}C_\tau)^*K_p=\overline{K_B(p)}K_{\tau(p)}.
$$
It follows that{\setlength\arraycolsep{2pt}
\begin{eqnarray*}
\|[C^*_\varphi, C_\varphi]k_p\|
&\ge & |\langle[C^*_\varphi, C_\varphi]k_p, k_p\rangle|
=\frac{| \| C_\varphi K_p\|^2-\|C^*_\varphi K_p\|^2|}{\|K_p\|^2}\\
&=&\frac{| \|\overline{K_B(p)}K_{\tau(p)}\|^2-\|K_{\varphi(p)}\|^2|}{\|K_p\|^2}\\
&=& \frac{\|K_{\varphi(p)}\|^2}{\|K_p\|^2}\biggl|1-\frac{|\overline{K_B(p)}|^2\|K_{\tau(p)}\|^2}{\|K_{\varphi(p)}\|^2}\biggr|.
\end{eqnarray*}}

If $\|A\|<1$, then $C_\varphi$ is compact and the result obviously holds. Thus,
it  suffices to prove that $C_\varphi$ is normal in the case of $\|A\|=1$. This means,  we  should prove that  $B=0$ and $A^*A=AA^*$ according to Theorem 3.1.
Now, assume $\|A\|=1$,  there exists a $\zeta\in\mathbb{C}^n$
such that $|A\zeta|=|\zeta|$.
Choosing   $p=t\zeta$ with $t\to \infty$, using Theorem A,
we find that
$$
\langle Ap, B\rangle=\langle A(t\zeta), B\rangle=t\langle A\zeta, B\rangle=0
$$
and $|Ap|^2=|t\zeta|^2=|p|^2$.
This implies{\setlength\arraycolsep{2pt}
\begin{eqnarray*}\frac{\|K_{\varphi(p)}\|^2}{\|K_{p}\|^2}
&=& \exp\biggl(\frac{|\varphi(p)|^2-|p|^2}{2}\biggr)
=\exp\biggl(\frac{|Ap+B|^2-|p|^2}{2}\biggr)\\
&=&\exp\biggl(\frac{|Ap|^2+2\mbox{Re}\langle Ap, B\rangle+|B|^2-|p|^2}{2}\biggr) \\
&=& \exp(|B|^2/2),
\end{eqnarray*}}
and{\setlength\arraycolsep{2pt}
\begin{eqnarray*}
\frac{|\overline{K_B(p)}|^2\|K_{\tau(p)}\|^2}{\|K_{\varphi(p)}\|^2}
&=& \exp\biggl(\frac{|\tau(p)|^2-|\varphi(p)|^2+2\mbox{Re}\langle p, B\rangle}{2}\biggl)
\\ &=& \exp\biggl(\frac{|A^*p|^2-|Ap+B|^2+2\mbox{Re}\langle p, B\rangle}{2}\biggl)
\\ &=&  \exp\biggl(\frac{|A^*p|^2-|p|^2-|B|^2+2\mbox{Re}\langle p, B\rangle}{2}\biggl).
\end{eqnarray*}}
Therefore,
{\setlength\arraycolsep{2pt}\begin{eqnarray*}
\|[C^*_\varphi, C_\varphi]k_{p}\|& \ge &\frac{\|K_{\varphi(p)}\|^2}{\|K_{p}\|^2}\biggl|1-\frac{|\overline{K_B(p)}|^2
\|K_{\tau(p)}\|^2}{\|K_{\varphi(p)}\|^2}\biggr|\\
&=&\exp(|B|^2/2)\biggl|1- \exp\biggl(\frac{|A^*p|^2-|p|^2-|B|^2+2\mbox{Re}\langle p, B\rangle }{2}\biggl)\biggr|.
\end{eqnarray*}}Because
$\limsup\limits_{t\to \infty}\|[C^*_\varphi, C_\varphi]k_{t\zeta}\|=0$,
we deduce that
$$
\lim_{t\to\infty}
\exp(|B|^2/2)\biggl|1-\exp\biggl(\frac{t^2|A^*\zeta|^2-t^2|\zeta|^2-|B|^2+2t\mbox{Re}\langle \zeta, B\rangle}{2}\biggl)\biggr|=0,
$$
that is,
$$
\lim_{t\to\infty}\exp\biggl(\frac{t^2|A^*\zeta|^2-t^2|\zeta|^2-|B|^2+2t\mbox{Re}\langle \zeta, B\rangle}{2}\biggl)=1.
$$
As a consequence, we must have $B=0$ and  $|A^*\zeta|= |\zeta|$.
This yields that  $\varphi(z)=Az$  and $C^*_\varphi =C_\tau$.

Now, we have
$$[C^*_\varphi, C_\varphi]=
C^*_\varphi C_\varphi-C_\varphi C^*_\varphi
=C_\tau C_\varphi-C_\varphi C_\tau
=C_{\varphi\circ\tau}-C_{\tau\circ\varphi}.
$$
Since $\|A\|=1$ implies that $\|A^*A\|=1$ and $\|AA^*\|=1$,
by Theorem A, both
$C_{\varphi\circ\tau}$ and $C_{\tau\circ\varphi}$ are not compact.
On the other hand, Choe et al. \cite{CCK} have shown that for holomorphic mappings $\psi \ne \phi$, the operator
$C_\psi-C_\phi$ is compact on $\mathcal{F}^2(\mathbb{C}^n)$ if and only if both  $C_\psi$ and $C_\phi$ are compact.
Thus, $[C^*_\varphi, C_\varphi]=C_{\varphi\circ\tau}-C_{\tau\circ\varphi}$ is compact must give that  $\varphi\circ\tau=\tau\circ\varphi$.
It follows then that $AA^*=A^*A$ and $C_\varphi$ is normal by Theorem 3.1. So we obtain the desired result.
\end{proof}

\section{Spectra of composition operators}\label{spectra}

When $\varphi$ is a holomorphic mapping of the complex plane $\mathbb{C}$,  Guo and Izuchi  \cite{GI} have described  the spectrum  of $C_\varphi$
on  $\mathcal{F}^2(\mathbb{C})$ as follows:
Let $\varphi(z)=az+b, |a|\le 1$ and $a\ne 1$, which inducing  a bounded composition operator $C_\varphi$ on  $\mathcal{F}^2(\mathbb{C})$.
Then $\sigma(C_\varphi)=\overline{\{a^n, n\in\mathbb{Z}_+\}}$, where  $\sigma(C_\varphi)$ denotes
the spectrum of $C_\varphi$.

In higher dimensions, we also want to know the spectral structure of $C_\varphi$ on $\mathcal{F}^2(\mathbb{C}^n)$.
In fact, when $\varphi$ is a unitary map, the spectrum of $C_\varphi$  is clear.
\\ \par
{\bf Theorem B. \cite{LZ}}
{\it Let $\varphi(z)=Uz$ with  $\{\lambda_1,\ldots,\lambda_n\}$ being   eigenvalues of the unitary matrix $U$. Then
the spectrum of $C_\varphi$  on $\mathcal{F}^2(\mathbb{C}^n)$ is the closure of the set
 $\{\lambda_1^{\alpha_1}\cdots \lambda_n^{\alpha_n}: (\alpha_1,\ldots, \alpha_n)\in \mathbb{N}^n\}$.}
\\ \par
In this section, we will completely give the spectrum of $C_\varphi$ for any bounded composition operator $C_\varphi$ on  $\mathcal{F}^2(\mathbb{C}^n)$.
The idea comes from  Bayart \cite{B10}. First, we need the following lemma.

\begin{lemma} \label{lem 4.1}
Assume that $A$ is an arbitrary $n\times n$ matrix with $\|A\|\le 1$. Then   there exists a unitary matrix $U\in\mathbb{C}^{n\times n}$ such that $UAU^*=M$ with
\begin{displaymath}
M=\left(
\begin{array}{cc}
D & 0
\\ 0 & A_1
\end{array}\right),
\end{displaymath}
where $D=\mbox{diag}(e^{i\theta_1},\ldots, e^{i\theta_s})$  and $A_1\in\mathbb{C}^{(n-s)\times (n-s)}$ is upper-triangular with all diagonal 
elements less than 1.
\end{lemma}

\begin{proof}
For any $A\in \mathbb{C}^{n\times n}$, by Schur Decomposition, there exist a unitary matrix $U\in  \mathbb{C}^{n\times n}$ and an upper triangular matrix $M\in \mathbb{C}^{n\times n}$ such that $$UAU^*=M.$$ Assume that $a_{11}, a_{22},\ldots, a_{nn}$ are the diagonal elements of $M$ with $|a_{11}|\ge | a_{22}|\ge \ldots\ge |a_{nn}|$. It is clear that  $a_{11}, a_{22},\ldots, a_{nn}$ are the eigenvalues of $M$. Moreover,
 $\|M\|=\|A\|\le 1$ gives $|a_{11}|\le 1$. If $\|M\|<1$, then $|a_{11}|< 1$ and  $s=0$.  The result is true.

If $\|M\|=1$, assume that $|a_{11}|= | a_{22}|= \cdots= |a_{ss}|=1$ for some $1\le s\le n$ and $M=(a_{jk})$. 
We will show that  $a_{ik}=0$ for $1\le i\le s, i<k\le n$.
Let $e_i=(0,\ldots,1,\ldots,0)^\perp$,  $i=1,\ldots, s$, the unit vectors of $\mathbb{C}^n$. Since $\|M^*\|=\|M\|=1$, we have $$|M^* e_i|^2=|a_{ii}|^2+\sum\limits_{i<k\le n}|a_{ik}|^2\le 1.$$
Combining this with  $|a_{ii}|=1$ $(i=1,\ldots, s)$, we  get $a_{ik}=0$ for $i<k\le n$. Therefore, $M$ has the form
\begin{displaymath}
M=\left(
\begin{array}{cc}
D & 0
\\ 0 & A_1
\end{array}\right),
\end{displaymath}
where $D=\mbox{diag}(a_{11}, a_{22},\ldots, a_{ss})$  and $A_1$ is  an upper-triangular matrix with the diagonal elements satisfying   $\max\limits_{s<i\le n}\{|a_{ii}|\}<1$.
\end{proof}

\begin{theorem} \label{the 4.2}
Let $\varphi$ be a holomorphic self-mapping of  $\mathbb{C}^n$. Suppose that $C_\varphi$ is bounded on $\mathcal{F}^2(\mathbb{C}^n)$.
Then $\varphi(z)=Az+B$ and the spectrum $C_\varphi$ is the closure of the set
$$
\{\lambda_1^{\alpha_1}\cdots\lambda_n^{\alpha_n}:\  (\alpha_1,\ldots,\alpha_n)\in\mathbb{N}^n\},
$$
where $\lambda_1,\ldots,\lambda_n$ are the eigenvalues of the matrix $A$.
\end{theorem}

\begin{proof}
Since $C_\varphi$ is bounded on $\mathcal{F}^2(\mathbb{C}^n)$, by Theorem A,
we have  $\varphi(z)=Az+B$ with $\|A\|\le 1$. Applying Lemma 4.1,
 there exists  unitary $U$ such that $UAU^*=M$ with
\begin{displaymath}
M=\left(
\begin{array}{cc}
D & 0
\\ 0 & A_1
\end{array}\right),
\end{displaymath}
where $D=\mbox{diag}(e^{i\theta_1},\ldots, e^{i\theta_s})$  and $A_1\in\mathbb{C}^{(n-s)\times (n-s)}$ is upper-triangular. Let   $\lambda_1,\ldots,\lambda_{n-s}$ be the eigenvalues of $A_1$, then $\lambda=\max\{\lambda_1,\ldots,\lambda_{n-s}\}<1$.

Now, we compute that {\setlength\arraycolsep{2pt}
\begin{eqnarray*}C^*_U C_\varphi C_U f(z)&=&f(U\circ\varphi\circ U^*(z))
\\ &=& f(UAU^*z+UB)=f(Mz+B')
\\ &=& C_\psi f(z).
\end{eqnarray*}}That is, $C_\varphi$ is similar to $C_\psi$ with $\psi(z)=Mz+B'$ and $B'=UB$.

Note that the boundedness of  $C_\varphi$  means that
 $C_\psi$ is also bounded.
Using Theorem A,
we have $\langle M\zeta, B'\rangle=0$ for any $\zeta\in\mathbb{C}^n$
with $|M\zeta|=|\zeta|$.
Choosing $\zeta=e_i$, $i=1,\ldots, s$,
we obtain that
$$
b_1'=\cdots=b_s'=0,
$$
where $b_i'$ is the $i^{th}$ coordinate of $B'$. Thus,
$$
\psi(z)=\psi(w, v)=(Dw, A_1v+B_1)$$  for  $z=(w, v)\in\mathbb{C}^s \times \mathbb{C}^{n-s}$.

Because the spectrum is similarly invariant, we will compute the spectrum of $C_\psi$.
Let $v(i)\in\mathbb{C}^{(n-s)\times 1}$
be a non-zero eigenvector of $A_1^\top$
associated to $\lambda_i$, i.e. 
 $A_1^\top v(i)=\lambda_i v(i)$, $i=1,\ldots, n-s$.
  Since $\lambda<1$, we see that $I-A_1$ is invertible. Choose the vector $C$
such that  $C=(I-A_1)^{-1}B$, that is, $B_1-C=-A_1C$.
Set the function
$$
F(z)=F(w, v)=w^\beta [(v-C)^\top v(1)]^{\gamma_1}\cdots[(v-C)^\top v(n-s)]^{\gamma_{n-s}},
$$ where
$\beta=(\beta_1,\ldots,\beta_s)\in \mathbb{N}^s$
and
$\gamma=(\gamma_1,\ldots,\gamma_{n-s})\in \mathbb{N}^{n-s}$.
Note that  $[(v-C)^\top v(i)]^{\gamma_i}$ is a  polynomial
of $v$ with degree $\gamma_i$ and  the Fock space $\mathcal{F}^2(\mathbb{C}^n)$ contains all polynomials, then
the function $F(w, v)$ is in  $\mathcal{F}^2(\mathbb{C}^n)$. Next, we compute that
\begin{eqnarray*}
&& C_\psi F(w, v)
=F\circ \psi(w, v)= F(D w, A_1v+B_1)
\\ &&=(Dw)^\beta [(A_1v+B_1-C)^\top v(1)]^{\gamma_1}\cdots[(A_1v+B_1-C)^\top
v(n-s)]^{\gamma_{n-s}}
\\ && = e^{i(\beta_1\theta_1+\cdots\beta_s\theta_s)}w^\beta
\{[A_1(v-C)]^\top v(1)\}^{\gamma_1}\cdots\{[A_1(v-C)]^\top
v(n-s)\}^{\gamma_{n-s}}
\\ && =e^{i(\beta_1\theta_1+\cdots\beta_s\theta_s)}w^\beta
[(v-C)^\top A_1^\top v(1)]^{\gamma_1}\cdots[(v-C)^\top A_1^\top
v(n-s)]^{\gamma_{n-s}}
\\ && = e^{i(\beta_1\theta_1+\cdots\beta_s\theta_s)}
\lambda_1^{\gamma_1}\cdots\lambda_{n-s}^{\gamma_{n-s}}w^\beta
[(v-C)^\top v(1)]^{\gamma_1}\cdots[(v-C)^\top v(n-s)]^{\gamma_{n-s}}
\\ && = e^{i(\beta_1\theta_1+\cdots\beta_s\theta_s)}
\lambda_1^{\gamma_1}\cdots\lambda_{n-s}^{\gamma_{n-s}}F(w, v),
\end{eqnarray*}
Thus,  for any multi-index $(\beta,
\gamma)\in \mathbb{N}^n$, the function $F(z, w)$ is an eigenvector of $C_\psi$ associated to the eigenvalue
$e^{i(\beta_1\theta_1+\cdots\beta_s\theta_s)}\lambda_1^{\gamma_1}\cdots\lambda_{n-s}^{\gamma_{n-s}}$.
Hence, the spectrum of $C_\psi$ contains the closure of the set
$$
\{e^{i(\beta_1\theta_1+\cdots\beta_s\theta_s)}\lambda_1^{\gamma_1}\cdots\lambda_{n-s}^{\gamma_{n-s}}:
\alpha=(\beta,\gamma)\in \mathbb{N}^n \},
$$
where $e^{i\theta_1},\ldots, e^{i\theta_s}$, $\lambda_1,\ldots,
\lambda_{n-s}$ are the eigenvalues of $M$.

In fact, we  find that the spectrum of  $C_\psi$ is exactly the closure of the above set.
Next, we will prove the other direction.

Let $\psi_N=\psi\circ\cdots\circ\psi$ ($N$ times). Note that $\psi(w, v)=(D w, A_1v+B_1)$ gives $$\psi_N(w, v)=(D^N w, A_1^Nv+B_N),$$
where $B_N=(A_1^{N-1}+\cdots +A_1+I)B_1=(I-A_1^N)(I-A_1)^{-1}B_1$. Moreover,
$D^N=\mbox{diag}(e^{iN\theta_1},\ldots, e^{iN\theta_s})$  and $A_1^N$ is still an upper-triangular matrix with the diagonal elements $\lambda_1^N,\ldots,\lambda_{n-s}^N$.
Since $\lambda=\max\{\lambda_1,\ldots,\lambda_{n-s}\}<1$, it is easy to check that $A_1^N\to O$ (here, $O$ denotes the zero matrix)
and $B_N\to (I-A_1)^{-1}B_1$ as $N\to \infty$.
Thus, we may choose an integer $N>0$ large enough such that $\|A_1^N\|<1$.

Using the spectral mapping theorem, we know that $[\sigma(C_\psi)]^N=\sigma[(C_\psi)^N]=\sigma(C_{\psi_N})$.
This leads to that   we may still use $\psi(w, v)=(D w, A_1v+B_1)$ with $\|A_1\|<1$
 instead of $\psi_N(w, v)=(D^N w, A_1^Nv+B_N)$ with $\|A^N_1\|<1$ to compute the spectrum of $\sigma(C_\psi)$.

Now, we introduce a decomposition for the Fock space  $\mathcal{F}^2(\mathbb{C}^n)$.
For $\gamma\in \mathbb{N}^{n-s}$, let $H_\gamma$ be the set of all
functions $F$ in  $\mathcal{F}^2(\mathbb{C}^n)$, where $F$ may be written as
$F_\gamma(w)v^\gamma$. Set $K_m=\bigoplus^\perp_{|\gamma|\ge m}H_\gamma$ for any integer $m\ge0$.
Thus, we have a finite  orthogonal decomposition
$
\mathcal{F}^2(\mathbb{C}^n)
=\bigoplus^\perp_{|\gamma|<m}H_\gamma\bigoplus^\perp K_m.
$

On the other hand, for the set $\mathbb{N}^{n-s}$, we need an order  (see \cite{B10}):
for $\alpha,\beta\in\mathbb{N}^{n-s}$,
\begin{eqnarray*}
\alpha<\beta\Longleftrightarrow
\begin{cases}|\alpha|<|\beta|,
\\  |\alpha|\le|\beta|, \alpha_1=\beta_1,\ldots,\alpha_j=\beta_j \,
\mbox{for}\, j<j_0\,\mbox{and}\,\alpha_{j_0}<\beta_{j_0}.
\end{cases}
\end{eqnarray*}
Since $A_1$ is upper-triangular, we have
\begin{eqnarray*}
C_{\psi} F(w, v)
&=& F\circ{\psi}(w,v)
=F_\gamma(D w)(A_1v+B_1)^\gamma
\\ &=& F_\gamma(D w)\prod\limits_{j=1}^{n-s}\biggl(\lambda_j v_j
+\sum\limits_{k>j}a_{jk}v_k+b_j\biggr)^{\gamma_j}.
\end{eqnarray*}
This implies that the representation matrix of $C_\psi$  will be upper-triangular when using the above order
for the  decomposition
$\mathcal{F}^2(\mathbb{C}^n)= \bigoplus^\perp_{|\gamma|<m}H_\gamma\bigoplus^\perp K_m$.

Set $\rho(w,v)=(D w,v)$ and $\widetilde{\psi}=(D w, \widetilde{A_1}v)$ with $\widetilde{A_1}=\mbox{diag}(\lambda_1,\ldots, \lambda_{n-s})$.
It is easy to check that the composition operator $C_{\widetilde{\psi}}$ is bounded  and has a diagonal
matrix in the decomposition of $\mathcal{F}^2(\mathbb{C}^n)$.
Let $T_\gamma$ and $S_\gamma$ respectively
denote the diagonal blocks of $C_{\widetilde{\psi}}$  and $C_{\rho}$ corresponding to $H_\gamma$.
Then
$$
T_\gamma=\lambda_1^{\gamma_1}\cdots\lambda_{n-s}^{\gamma_{n-s}}S_\gamma.
$$
Note that $C_\rho$ is unitary on $\mathcal{F}^2(\mathbb{C}^n)$,
by Theorem  B, the spectrum of $C_\rho$ is the closure of the set
$$
\{e^{i(\beta_1\theta_1+\cdots\beta_s\theta_s)}:
(\beta_1,\ldots,\beta_s)\in \mathbb{N}^s \}.
$$
Using  Lemma 7.17 of \cite{CM1}, we get $\sigma(S_\gamma)\subset \sigma(C_\rho)$ and
\begin{eqnarray*}\sigma(C_{\widetilde{\psi}})& \subset &\bigcup\limits_{|\gamma|<m}
\sigma(T_\gamma)\bigcup \sigma(C_{\widetilde{\psi} |K_m})
\\ &\subset  &\bigcup\limits_{|\gamma|<m} \lambda_1^{\gamma_1}\cdots
\lambda_{n-s}^{\gamma_{n-s}}\sigma(S_\gamma)\bigcup \sigma(C_{\widetilde{\psi} |K_m}).
\end{eqnarray*}Thus, we  obtain that the spectrum of $C_{\widetilde{\psi}}$ is in  the closure of
$$
\{e^{i(\beta_1\theta_1+\cdots\beta_s\theta_s)}\lambda_1^{\gamma_1}\cdots\lambda_{n-s}^{\gamma_{n-s}}:
\alpha=(\beta,\gamma)\in \mathbb{N}^n \}.
$$

On the other hand, since the matrix of  $C_\psi$ becomes  upper-triangular in the decomposition of $\mathcal{F}^2(\mathbb{C}^n)$ we have designed before,
 the spectra for the diagonal locks of $C_\psi$ corresponding to the subspace  $H_\gamma$ are equal to those of $C_{\widetilde{\psi}}$.
We also let  $T_\gamma$
denote the diagonal blocks of $C_{\psi}$ corresponding to   $H_\gamma$.
Using   Lemma 7.17 of \cite{CM1} again,
\begin{eqnarray*}
\sigma(C_\psi)& \subset &\bigcup\limits_{|\gamma|<m}
\sigma(T_\gamma)\bigcup \sigma(C_{\psi |K_m})
\\ & \subset  & \bigcup\limits_{|\gamma|<m}\sigma(T_\gamma)\bigcup D(0, \|C_{\psi |K_m}\|),
\end{eqnarray*}
where $D(0, \|C_{\psi |K_m}\|)$ denotes  the disk with the radius $\|C_{\psi |K_m}\|$.
Therefore, if we show that $\|C_{\psi |K_m}\|$ tends to
zero as $m\to\infty$,  then $C_{\psi}$ and $C_{\widetilde{\psi}}$ have the same spectrum.

Let $A_1=T\Sigma V$ be a singular value decomposition of $A_1$, where
$T, V$ are unitary matrices of $\mathbb{C}^{(n-s)\times(n-s)}$ and $\Sigma=\mbox{diag}(\mu_1, \ldots, \mu_{n-s})$ with $\mu_1\ge \ldots\ge \mu_{n-s}$,
the non-negative square roots of the eigenvalues of $A_1^\ast A_1$.
Now $\|A_1\|<1$ yields that
$\mu_1=\|A_1\|<1$. Assume $\mu_1<\mu<1$ for a positive  constant $\mu$. Let $\phi_1(w, v)=(Dw,Vv)$,
$\psi_\Sigma(w, v)=(w, \Sigma v+T^*B_1)$ and $\phi_2(w, v)=(w, Tv)$,
then $C_\psi=C_{\phi_1}C_{\psi_\Sigma}C_{\phi_2}$.
Obviously, $C_{\phi_1}$ and $C_{\phi_2}$ are unitary operators when acting on $\mathcal{F}^2(\mathbb{C}^n)$. Furthermore,
$C_{\phi_2}$ preserves both $K_m$ and $K_m^\bot$.
Hence, it is sufficient to  prove  that $\|C_{\psi_\Sigma |K_m}\|\to 0$ as $m\to \infty$.

For $F(z)=\sum\limits_{|\gamma|\ge m}  F_\gamma(w)v^\gamma\in K_m$, we compute that{\setlength\arraycolsep{2pt}
\begin{eqnarray*}
\|C_{\psi_\Sigma} F\|^2
&=&\int_{\mathbb{C}^n} |F\circ\psi_\Sigma(w,v)|^2 e^{-\frac{1}{2}(|w|^2+|v|^2)}dwdv\\
\\
&\le & \sum\limits_{|\gamma|\ge m} \int_{\mathbb{C}^n}
|F_\gamma(w)(\Sigma v+T^*B_1)^\gamma|^2e^{-\frac{1}{2}(|w|^2+|v|^2)}dwdv\\
&= & \sum\limits_{|\gamma|\ge m} \int_{\mathbb{C}^s}|F_\gamma(w)|^2
e^{-\frac{1}{2}|w|^2}dw\int_{\mathbb{C}^{n-s}}|(\Sigma v+T^*B)^\gamma |^2 e^{-\frac{1}{2}|v|^2}dv
\end{eqnarray*}}Let $T^*B=(b_1,\ldots, b_{n-s})$, then $\Sigma v+T^*B=(\mu_1v_1+b_1,\ldots,\mu_{n-s}v_{n-s}+b_{n-s})$.
Since $\mu_1<\mu<1$, for $i=1,\ldots, n-s$, it is easy to see $|\mu_iv_i+b_i|\le \mu|v_i|$ when $|v_i|\ge M$ for large enough $M>0$.
It follows that $|(\Sigma v+T^*B)^\gamma |\le \mu^m|v^\gamma|$ off a compact subset of $\mathbb{C}^{n-s}$.
Therefore, using  the orthogonality of
$F_\gamma(w)v^\gamma$ and $F_{\gamma'}(w)v^{\gamma'}$ for $\gamma\ne \gamma'$,
{\setlength\arraycolsep{2pt}
\begin{eqnarray*}
\|C_{\psi_\Sigma} F\|^2
&\le &C\mu^{2m} \sum\limits_{|\gamma|\ge m} \int_{\mathbb{C}^s}|F_\gamma(w)|^2
e^{-\frac{1}{2}|w|^2}dw\int_{\mathbb{C}^{n-s}}|v^\gamma|^2 e^{-\frac{1}{2}|v|^2}dv\\
&= & C\mu^{2m}\int_{\mathbb{C}^n}\biggl| \sum\limits_{|\gamma|\ge m} F_\gamma(w)v^\gamma\biggr|^2
e^{-\frac{1}{2}(|w|^2+|v|^2)}dwdv
\\ &=& C\mu^{2m}  \|F\|^2,
\end{eqnarray*}}where $C$ is a sufficiently large constant.
Now $\mu<1$ gives
$$
\lim_{m\to \infty}\|C_{\psi_\Sigma |K_m}\|\le \lim\limits_{m\to \infty}C\mu^{2m}=0.
$$
The desired result holds.
\end{proof}

\section{Cyclicity of composition operators}

A bounded linear operator $T$ on a linear metric space $\mathcal{H}$
is said to be cyclic if there exists a vector $x\in\mathcal{H}$ such that
$$
\overline{\mbox{span}\{T^mx: m=0,1,\ldots\}}=\mathcal{H}.
$$
If  there exists a vector $x\in\mathcal{H}$ such that the orbit
$$
\{T^mx: m=0,1,\ldots\}
$$
is dense, then $T$ is said to be hypercyclic.
If there exists a vector $x\in\mathcal{H}$ such that the projective orbit
$$
\{\lambda T^mx: m=0,1,\ldots\,\mbox{and}\, \lambda\in \mathbb{C}\}
$$
is dense, then $T$ is said to be supercyclic. See \cite{BS97} for more information.

In this section, we study the dynamics of composition operators
on the Fock space $\mathcal{F}^2(\mathbb{C}^n)$. On the complex plane, the following result has been proved by Guo and Izuchi \cite{GI}.

\begin{proposition}\label{pro 5.1}
(a) If $\varphi(z)=az$ with $|a|=1$, then $C_\varphi$ is cyclic on the Fock space
$\mathcal{F}^2(\mathbb{C})$ if and only if $a^n\ne a$ for every $n>1$.
\\ (b)  If $\varphi(z)=az+b$ with $|a|<1$ and $a\ne 0$, then $C_\varphi$ is cyclic on $\mathcal{F}^2(\mathbb{C})$.
\end{proposition}

\begin{proof}
(a) This result is Proposition 2.3(i) of \cite{GI}.

(b) This is an immediate result of Theorem 4.2 in \cite{GI}. Here, we present a different proof. Let $K_z(w)=\exp(\langle w,z\rangle/2)$ be the reproducing kernels for
$\mathcal{F}^2(\mathbb{C})$.
We will prove that $K_z$ is a cyclic vector
of $C_\varphi$ for any $z\ne 0$. We see that{\setlength\arraycolsep{2pt}
\begin{eqnarray*}
C^m_\varphi K_z(w) &=& K_z(\varphi_m(w))=K_z\biggl(a^mw+\frac{1-a^m}{1-a}b\biggr)
\\ &=& \exp\biggl(\biggl\langle a^mw+\frac{1-a^m}{1-a}b, z\biggr\rangle/2\biggr)
\\ &=& \exp(\langle w, \overline{a^m}z\rangle/2)\exp\biggl(\biggl\langle\frac{1-a^m}{1-a}b, z\biggr\rangle/2\biggr)
\\ &=&c_m K_{\overline{a}^mz}(w),
\end{eqnarray*}}where $c_m=\exp\biggl(\biggl\langle\frac{1-a^m}{1-a}b, z\biggr\rangle/2\biggr)$ is a constant.
If $f$ is orthogonal to the set
$$
\{C_\varphi^mK_z: m=0,1,\ldots\},
$$
then
$$
0=\langle f,C_\varphi^mK_z\rangle
=\langle f, c_m K_{\overline{a}^mz}\rangle
=\overline{c_m}f(\overline{a}^mz).
$$
Since $|a|<1$,
for any fixed $z\ne 0$, we have $\overline{a}^mz\to 0$  and
$$
c_m=\exp\biggl(\biggl\langle\frac{1-a^m}{1-a}b, z\biggr\rangle/2\biggr)
\to \exp\biggl(\biggl\langle\frac{b}{1-a}, z\biggr\rangle/2\biggr)
$$
as $m\to\infty$.
Thus, $f$ vanishes on a sequence of points on $\mathbb{C}$ with limit  $0$,
which implies that $f\equiv 0$ in $\mathbb{C}$.
Hence,  $C_\varphi$ is cyclic with the cyclic vector $K_z$ for any $z\ne 0$.
\end{proof}

It is easy to see that $\varphi$ must fix a point of $\mathbb{C}$ when $C_\varphi$
is bounded on the Fock space $\mathcal{F}^2(\mathbb{C})$.
Applying the same technique as Theorem 5.2 in \cite{GM04},
we obtain that $C_\varphi$ is not supercyclic on $\mathcal{F}^2(\mathbb{C})$.
Therefore, the cyclicity of composition operators on $\mathcal{F}^2(\mathbb{C})$ is very simple.

In order to describe the dynamics of composition operators on the Fock space$\mathcal{F}^2(\mathbb{C}^n)$,
we first give the following characterization for those symbols which inducing
bounded composition operators on $\mathcal{F}^2(\mathbb{C}^n)$.

\begin{lemma}\label{lem 5.2}
If $\varphi(z)=Az+B$ induces a bounded composition operator on $\mathcal{F}^2(\mathbb{C}^n)$, then $\varphi$ fixes a point in $\mathbb{C}^n$.
\end{lemma}

\begin{proof}
It suffices to show $B$  belongs to the orthogonal complement of $\mbox{ker}(I-A^*)$. Suppose that $v$
 is a unit vector in $\mbox{ker}(I-A^*)$. Then $A^*v = v$ and since
$\|A\|\le 1$, we get $Av = v$ as well:
$$1\ge |Av| = |Av| |v| \ge |\langle Av, v\rangle |=  |\langle A A^*v, v\rangle | =   |\langle A^*v, A^*v\rangle | =|\langle v, v\rangle | = 1;$$
thus, $|Av|=|v|=|\langle Av, v\rangle |$ and we conclude that $Av=\lambda v$ for some constant $\lambda$. Finally, $\lambda=\langle \lambda v, v\rangle=\langle A v, v\rangle=\langle v, A^* v\rangle=\langle v, v\rangle=1$. In particular, we
have $|Av|= |v|$ and hence,
$$\langle B, v\rangle=\langle B, Av\rangle=0$$
since $C_\varphi$ is bounded.
\end{proof}

First, similar to the proof of Theorem 3.3 in \cite{JO08}, we obtain  a  necessary and sufficient condition for
unitary composition operators to be cyclic on $\mathcal{F}^2(\mathbb{C}^n)$.  Here, we omit its proof.

\begin{theorem}\label{the 5.3}
If $\varphi(z)=Uz$ and  $U$ is unitary with the eigenvalues
$e^{i\theta_1},\ldots, e^{i\theta_n}$, then $C_\varphi$ is cyclic on $\mathcal{F}^2(\mathbb{C}^n)$
if and only if $\theta_1,\ldots,\theta_n,\pi$ are rationally linearly independent.
\end{theorem}

For all  bounded composition operators on the Fock space $\mathcal{F}^2(\mathbb{C}^n)$, the following result shows that their are not supercyclic.

\begin{theorem}\label{the 5.4}
Let $\varphi:\mathbb{C}^n\rightarrow\mathbb{C}^n$ be  a holomorphic mapping. If $C_\varphi$ is bounded  on $\mathcal{F}^2(\mathbb{C}^n)$,
 then $C_\varphi$ is not supercyclic.
\end{theorem}

\begin{proof}
We will use the similar idea as  Theorem 5.2 in \cite{GM04} to obtain our result.

As shown in Lemma 5.2, $\varphi(z)=Az+B$ must have a fixed point $p$.
Suppose that $f$ is a supercyclic vector for $C_\varphi$.
It is clear that we must have $f(p)\ne 0$.
Assume that $f(p)=1$.
If the function $g\in\mathcal{F}^2(\mathbb{C}^n)$ is in the projective orbit of $f$ under $C_\varphi$,
then there exists a sequence $\{\lambda_{n_k}\}$ such that $\{\lambda_{n_k} C_{\varphi_{n_k}}f\}$
tends to $g$ as $k\to\infty$.
Since norm convergence implies pointwise convergence,
we get
$$
g(p)=\lim\limits_{k\to\infty}\lambda_{n_k} C_{\varphi_{n_k}}f(p)
=\lim\limits_{k\to\infty}\lambda_{n_k}f(\varphi_{n_k}(p))
=\lim\limits_{k\to\infty}\lambda_{n_k}f(p)
=\lim\limits_{k\to\infty}\lambda_{n_k}.
$$

If $\varphi(z)=Uz$.
Let $e^{i\theta_1},\ldots, e^{i\theta_n}$ be the eigenvalues  of $U$,
then  $\theta_1,\ldots,\theta_n,\pi$ are rationally linearly independent.
Otherwise, using Theorem 5.3, $C_\varphi$ is not cyclic, and then not supercyclic on $\mathcal{F}^2(\mathbb{C}^n)$.
Therefore, by extracting a subsequence, we may assume that the sequence
$\{\varphi_{n_k}(z)\}=\{U^{n_k}z\}$ converges to a map $\phi(z)=Vz$ with unitary $V$.
Now, we choose a univalent function $g$ with $g(p)\ne 0$.
For each $z\in \mathbb{C}^n$, we have
\begin{eqnarray*}
g(z)&=&\lim\limits_{k\to\infty}\lambda_{n_k} C_{\varphi_{n_k}}f(z)
=\lim\limits_{k\to\infty}\lambda_{n_k}f(\varphi_{n_k}(z))
\\ &=& g(p)\lim\limits_{k\to\infty}f(\varphi_{n_k}(z))=g(p)f(Vz).
\end{eqnarray*}
This yields that $f(z)=g\circ V^{-1}(z)/g(p)$ is univalent.
It follows that all the scalar multiplies of the $C_\varphi$ orbit of $f$ are univalent functions.
This means that  $C_\varphi$ can not be supercyclic.

For more general $\varphi(z)=Az+B$, it is enough to prove that $C_\psi$ is not  supercyclic, where $\psi(w, v)=(Dw, A_1v+B_1)$
is described in the proof of Theorem 4.2.
Note that $\psi_m(w, v)=(D^mw, A^m_1v+(A^{m-1}_1+\cdots+I)B_1)$.
Let $E$ be any compact subset of $\mathbb{C}^{n-s}$ and write
$0\times E=\{(0, v)\in \mathbb{C}^s\times\mathbb{C}^{n-s}, v\in E\}$.
Since $\lambda=\max\{\lambda_1,\ldots,\lambda_{n-s}\}<1$, as pointed in  the proof of Theorem 4.2,
  the sequence of iterates of $\psi$ tends to the point
$(0, (I-A_1)^{-1}B_1)$ uniformly on the set $0\times E$ of $\mathbb{C}^n$. Moreover, it is clear that $(0, (I-A_1)^{-1}B_1)$ is
exactly the fixed point $p$ of $\psi$.
It follows that
$$
g(0,v)
=\lim\limits_{k\to\infty}\lambda_{n_k} C_{\psi_{n_k}}f(0,v)
=\lim\limits_{k\to\infty}\lambda_{n_k}f(\psi_{n_k}(0,v))
=g(p)f(p)=g(p)
$$
for any $(0, v)\in 0\times E$.
Consequently, only functions which are independent with the last $n-s$ coordinates can be
in the closure of the projective $C_\psi$ orbit of $f$.
Therefore, $C_\psi$ is not supercyclic and the proof is complete.
\end{proof}

Finally, it is clear that $C_\varphi$ is not cyclic on the Fock space $\mathcal{F}^2(\mathbb{C}^n)$,
if $\varphi(z)=Az+B$ and the matrix $A$ is not invertible.  Now, we almost  have known the dynamics of bounded composition operators on  $\mathcal{F}^2(\mathbb{C}^n)$,
only   the following problem still open.
\\ \par
{\bf Problem:} For the mapping $\varphi(z)=Az+B$, if $A$ is invertible with $\|A\|\le 1$
and $C_\varphi$ is bounded on $\mathcal{F}^2(\mathbb{C}^n)$, is it cyclic?

\bigskip

\centerline{ ACKNOWLEDGEMENTS }

The first author  would like to think professor Zhihua Chen for helping prove Lemma 4.1.
The  authors also think other mathematicians for pointing out Proposition  3.2 and providing many good suggestions to improve this paper.

\end{document}